\documentclass[12pt]{amsart}
\usepackage{amsmath,amssymb,amsthm,amscd,esint}

\usepackage{color}

\newtheorem{theorem}{Theorem}[section]
\newtheorem{lemma}[theorem]{Lemma}

\newtheorem{corollary}[theorem]{Corollary}

\theoremstyle{definition}
\newtheorem{definition}[theorem]{Definition}

\theoremstyle{remark}
\newtheorem{remark}[theorem]{Remark}

\numberwithin{equation}{section}
\numberwithin{theorem}{section}

\DeclareMathOperator{\vol}{vol} \DeclareMathOperator{\dist}{dist}

\DeclareMathOperator{\diam}{diam}

\newcommand{\Ric}{\mathrm{Ric}}

\newcommand{\be}{\begin{equation}}
\newcommand{\ee}{\end{equation}}
\newcommand{\ba}{\begin{eqnarray}}
\newcommand{\ea}{\end{eqnarray}}
\newcommand{\ban}{\begin{eqnarray*}}
\newcommand{\ean}{\end{eqnarray*}}

\newcommand{\IN}{\mathrm{IN}}

\newcommand{\SN}{\mathrm{SN}}

\begin{document}
\title[Isoperimetric Constant]{Neumann Isoperimetric Constant Estimate \\ for 
convex domains }
\author{Xianzhe Dai}
\address[Xianzhe Dai]{Department of Mathematics, ECNU, Shanghai and UCSB, Santa Barbara CA 93106 \\ email:dai@math.ucsb.edu}
\author{Guofang Wei}
\address[Guofang Wei]{Department of Mathematics, University of California, Santa Barbara CA 93106 \\ email: wei@math.ucsb.edu}
\author{Zhenlei Zhang}
\address[Zhenlei Zhang]{Department of Mathematics, Capital Normal University, China \\ email: zhleigo@aliyun.com}
\date{}
\thanks{XD is partially supported by the Simons Foundation, NSF and NSFC, GW partially supported by the Simons Foundation and NSF DMS 1506393, ZZ partially supported by CNSF11371256 }
\begin{abstract}
We present a new and direct proof of the local Neumann isoperimetric inequality on convex domains of a Riemannian manifold with Ricci curvature bounded below.
\end{abstract}
\maketitle

\section{Introduction}

Isoperimetric and Sobolev inequalities are equivalent inequalities (see e.g. Theorem~\ref{ID-SD} below) which play important role in geometric analysis on manifolds. Indeed, doing analysis on manifolds usually depends on the estimate of the Sobolev constant which could then be obtained via the isoperimetric constant. For closed manifolds there are extensive work on  isoperimetric constant estimates. For domains, one usually needs to distinguish between two types of isoperimetric constants, the Dirichlet and the Neumann type. A  subtle point is that the
Neumann type isoperimetric inequality is harder to obtain and requires extra conditions on the boundary. For star-shaped domain in a manifold with Ricci curvature bounded from below, Buser \cite{Buser1982} obtained a Neumann  isoperimetric constant (the Cheeger constant) estimate (depends on in and out radius) using comparison geometry. For  domain with smooth and convex boundary, the known estimate for Neumann isoperimetric constant is very analytic and surprisingly indirect. Namely it is  through Li-Yau gradient estimate for heat kernel \cite{Li-Yau1986} and the equivalence of heat kernel bounds,  Sobolev inequality, isoperimetric inequality,see \cite[Page 448]{Saloff-Coste1992}. In this short note we give a geometric and direct proof of  a Neumann  isoperimetric inequality for convex domains (whose boundaries need not be smooth).  

First we recall some definitions.
\begin{definition}
When $M$ is compact (with or without boundary), the Neumann $\alpha$-isoperimetric constant of $M$ is defined by
\[
\IN_\alpha (M) = \sup_{H}  \frac{ \min \{ \vol (M_1), \vol (M_2) \}^{1-\frac{1}{\alpha}} }{\vol (H)}, \]
where $H$ varies over compact $(n-1)$-dim submanifold of $M$ which divides $M$ into two disjoint open submanifolds $M_1, M_2$ (with or without boundary).
\end{definition}

\begin{definition}
The Neumann $\alpha$-Sobolev constant of $M$ is defined by
\[
\SN_\alpha (M) = \sup_{f\in C^\infty(M)} \frac{\inf_{a \in \mathbb R}\|f-a\|_{\frac{\alpha}{\alpha-1}} }{ \|\nabla f\|_1}. \]
\end{definition}

The isoperimetric constant and Sobolev constant are equivalent.

\begin{theorem}[\cite{Cheeger1970}, see also \cite{Li}]  \label{ID-SD}
For all $n \le \alpha \le \infty$,
\[
\IN_\alpha (M) \ge \SN_\alpha (M) \ge  \frac 12 \, \IN_\alpha (M). \]
\end{theorem}

For convenience we consider the normalized Neumann $\alpha$-isoperimetric and $\alpha$-Sobolev constant:
$$\IN^*_\alpha (M) = \IN_\alpha (M) \, \vol (M)^{1/\alpha}, \ \ \  \ \SN^*_\alpha (M) = \SN_\alpha (M) \, \vol (M)^{1/\alpha}. $$

Using comparison geometry and Vitali covering we give an estimate on the normalized Neumann isoperimetric constant for convex domain in terms of the Ricci curvature lower bound and the diameter of the domain.  

\begin{theorem}\label{Thm}
Let $(M,g)$ be a complete Riemannian manifold of dimension $n$, with $\Ric\ge-(n-1)K$ for some $K\ge 0$. Let $\Omega$ be a bounded convex domain. Then
\begin{equation}
\IN^*_n(\Omega)\le 40^ne^{11(n-1)\sqrt{K}d}\cdot d
\end{equation}
where $d$ is the diameter of the domain $\Omega$.  In particular, if $M$ is closed with diameter $d$, then
\begin{equation}
\IN^*_n(M)\le 40^ne^{11(n-1)\sqrt{K}d}\cdot d.
\end{equation}
\end{theorem}

\begin{corollary}
Let $(M,g)$ be a complete Riemannian manifold of dimension $n$, with nonnegative Ricci curvature. Let $\Omega$ be a bounded convex domain. Then
\begin{equation}
\IN^*_n(\Omega)\le 40^n\cdot d
\end{equation}
where $d$ is the diameter of the domain $\Omega$. In particular, if $M$ is closed with diameter $d$, then
\begin{equation}
\IN^*_n(M)\le 40^n\cdot d.
\end{equation}
\end{corollary}

\begin{remark}
The case when $\Omega$ equals the whole manifold is well-known. 
The reference we mentioned earlier for convex domain in the literature deals with domains with (smooth) convex boundary which is a stronger condition. 
\end{remark}

\begin{remark}
For balls we can obtain both Dirichlet and Neumann isoperimetric constant estimates even under the much weaker  integral Ricci lower bound assumption \cite{DWZ16, Zhang}. On the other hand it is not clear if that will remain true for convex domains. 
\end{remark}
\begin{remark}
Using the mean curvature estimate from \cite{WW} one gets similar estimate  when the Bakry-Emery Ricci curvature is bounded from below and oscillation of the potential function is bounded. 
\end{remark}

%
%

\section{Proof of  Theorem~\ref{Thm}}

The proof goes by a covering argument of Anderson \cite{An92}, combined with an observation of Gromov \cite{Gr}. See \cite{An92} or \cite{DWZ16} for a similar argument of estimating the local Dirichlet isoperimetric constant. First of all we recall a lemma whose proof is a slight modification of Gromov's observation \cite[5.(C)]{Gr}.

\begin{lemma}
Let $M^n$ be a complete Riemannian manifold. Let $\Omega$ be a convex domain of $M$ and $H$ be any hypersurface dividing $\Omega$ into two parts $\Omega_1,\Omega_2$. For any Borel subsets $W_i\subset \Omega_i$, there exists $x_1$ in one of $W_i$, say $W_1$, and a subset $W$ in another one, $W_2$, such that
\begin{equation}
\vol(W)\ge\frac{1}{2}\vol(W_2)
\end{equation}
and any $x_2\in W$ has a unique minimal geodesic connecting to $x_1$ which intersects $H$ at some $z$ such that
\begin{equation}
\dist(x_1,z)\ge\dist(x_2,z).
\end{equation}
\end{lemma}

The convexity assumption of $\Omega$ is essential. It implies that any minimal geodesic with endpoints in different parts must intersects $H$. The Bishop-Gromov relative volume comparison theorem gives

\begin{lemma}
Let $H$, $W$ and $x_1$ be as in the lemma above. Then
\begin{equation}
\vol(W)\le 2^{n-1}De^{(n-1)\sqrt{K}D}\vol(H')
\end{equation}
where $D=\sup_{x\in W}\dist(x_1,x)$ and $H'$ is the set of intersection points with $H$ of geodesics $\gamma_{x_1,x}$ for all $x\in W$.
\end{lemma}
\begin{proof}
Let $\Gamma\subset S_{x_1}$ be the set of unit vectors such that $\gamma_v=\gamma_{x_1,x_2}$ for some $x_2\in W$. We compute the volume in the polar coordinate at $x_1$. Write $dv=\mathcal{A}(\theta,t)d\theta\wedge dt$ in the polar coordinate $(\theta,t)\in S_{x_1}\times\mathbb{R}^+$. For any $\theta\in\Gamma$, let $r(\theta)$ be the radius such that $\exp_{x_1}(r\theta)\in H$. Then $W\subset\{\exp_{x_1}(r\theta)|\theta\in\Gamma,\,r(\theta)\le r\le 2r(\theta)\}$. So, by relative volume comparison,
\begin{eqnarray}
\vol(W)&\le&\int_\Gamma\int_{r(\theta)}^{2r(\theta)}\mathcal{A}(\theta,t)dtd\theta\nonumber\\
&\le&\frac{\sinh^{n-1}(2\sqrt{K}D)}{\sinh^{n-1}(\sqrt{K}D)}\int_{\Gamma}r(\theta)\mathcal{A}(\theta,r(\theta))
d\theta \nonumber\\
&\le&D\frac{\sinh^{n-1}(2\sqrt{K}D)}{\sinh^{n-1}(\sqrt{K}D)}\vol(H')\nonumber.
\end{eqnarray}
The required estimate follows from $\frac{\sinh(2t)}{\sinh t}=2\cosh t\le e^t$ whenever $t\ge 0$.
\end{proof}

\begin{corollary}
Let $H$ be any hypersurface dividing a convex domain $\Omega$ into two parts $\Omega_1$, $\Omega_2$. For any ball $B=B_r(x)$ we have
\begin{eqnarray}
\min\big(\vol(B\cap \Omega_1),\vol(B\cap \Omega_2)\big)\le 2^{n+1}re^{(n-1)\sqrt{K}d}\vol(H\cap B_{2r}(x))
\end{eqnarray}
where $d=\diam(\Omega)$. In particular, if $B\cap \Omega$ is divided equally by $H$, we have
\begin{eqnarray}  \label{iso-equ-ball}
\vol(B_r(x)\cap\Omega)\le 2^{n+2}re^{(n-1)\sqrt{K}d}\vol(H\cap B_{2r}(x))
\end{eqnarray}
\end{corollary}
\begin{proof}
Put $W_i=B\cap \Omega_i$ in the above lemma and notice that $D\le2r$ and $H'\subset H\cap B_{2r}(x)$.
\end{proof}

Now we are ready to prove our main theorem.

\begin{proof}[Proof of Theorem \ref{Thm}]
We may assume that $\vol(\Omega_1)\le\vol(\Omega_2)$. For any $x\in\Omega_1$, let $r_x$ be the smallest radius such that
$$\vol(B_{r_x}(x)\cap\Omega_1)=\vol(B_{r_x}(x)\cap\Omega_2)=\frac{1}{2}\vol(B_{r_x}(x)\cap\Omega).$$
Let $d=\diam(\Omega)$. By above corollary,
\begin{equation}\label{hypersurface: 11}
\vol(B_{r_x}(x)\cap\Omega)\le 2^{n+2}r_xe^{(n-1)\sqrt{K}d}\vol(H\cap B_{2r}(x)).
\end{equation}
The domain $\Omega_1$ has a covering
$$\Omega_1\subset\bigcup_{x\in\Omega_1}B_{2r_x}(x).$$
By Vitali Covering Lemma, cf. \cite[Section 1.3]{LiYa}, we can choose a countable family of disjoint balls $B_i=B_{2r_{x_i}}(x_i)$ such that $\cup_i B_{10r_{x_i}}(x_i) \supset \Omega_1$.
Applying the relative volume comparison theorem and the convexity of $\Omega$ we have
\begin{eqnarray}
\vol(\Omega_1)&\le&\sum_i\frac{\int_0^{10r_{x_i}}\sinh^{n-1}(\sqrt{K}t)dt}{\int_0^{r_{x_i}}\sinh^{n-1}(\sqrt{K}t)dt}
\vol\big(B_{r_{x_i}}(x_i)\cap\Omega_1\big)\nonumber\\
&\le& 10\sum_i\frac{\sinh^{n-1}(10\sqrt{K}r_{x_i})}{\sinh^{n-1}(\sqrt{K}r_{x_i})}\vol\big(B_{r_{x_i}}(x_i)\cap\Omega_1\big)\nonumber\\
&\le&10\frac{\sinh^{n-1}(10\sqrt{K}d)}{\sinh^{n-1}(\sqrt{K}d)}\sum_i\vol\big(B_{r_{x_i}}(x_i)\cap\Omega_1\big)\nonumber\\
&\le&10^ne^{9(n-1)\sqrt{K}d}\sum_i\vol\big(B_{r_{x_i}}(x_i)\cap\Omega_1\big)\nonumber\\
&=&2^{-1}\cdot10^n\cdot e^{9(n-1)\sqrt{K}d}\sum_i\vol\big(B_{r_{x_i}}(x_i)\cap\Omega\big)\nonumber.
\end{eqnarray}
Moreover, since the balls $B_i$ are disjoint, (\ref{hypersurface: 11}) gives,
\begin{equation}\nonumber
\vol(H)\ge\sum_i\vol(B_i\cap H)\ge2^{-n-2}e^{-(n-1)\sqrt{K}d}\sum_ir_{x_i}^{-1}\vol(B_{r_{x_i}}(x_i)\cap\Omega).
\end{equation}
These two estimates lead to
\begin{eqnarray}
\frac{\vol(\Omega_1)^{\frac{n-1}{n}}}{\vol(H)}
&\le&2\cdot 20^{n}e^{10(n-1)\sqrt{K}d}\frac{\big(\sum_i\vol(B_{r_{x_i}}(x_i)\cap\Omega)\big)^{\frac{n-1}{n}}}
{\sum_ir_{x_i}^{-1}\vol(B_{r_{x_i}}(x_i)\cap\Omega)}\nonumber\\
&\le& 40^{n}e^{10(n-1)\sqrt{K}d}\frac{\sum_i\vol(B_{r_{x_i}}(x_i)\cap\Omega)^{\frac{n-1}{n}}}
{\sum_ir_{x_i}^{-1}\vol(B_{r_{x_i}}(x_i)\cap\Omega)}\nonumber\\
&\le&40^{n}e^{10(n-1)\sqrt{K}d}\sup_i\frac{\vol(B_{r_{x_i}}(x_i)\cap\Omega)^{\frac{n-1}{n}}}
{r_{x_i}^{-1}\vol(B_{r_{x_i}}(x_i)\cap\Omega)}\nonumber\\
&=&40^{n}e^{10(n-1)\sqrt{K}d}\sup_i\bigg(\frac{r_{x_i}^n}{\vol(B_{r_{x_i}}(x_i)\cap\Omega)}\bigg)^{\frac{1}{n}}.\nonumber
\end{eqnarray}
On the other hand, since  $\vol(\Omega_1)\le\vol(\Omega_2)$, we have $r_x\le d$ for any $x\in\Omega_1$. Thus, by the relative volume comparison and convexity of $\Omega$ again, we have
$$\vol(\Omega)\le\frac{\int_0^d\sinh^{n-1}(\sqrt{K}t)dt}{\int_0^{r_x}\sinh^{n-1}(\sqrt{K}t)dt}\vol(B_{r_x}(x)\cap\Omega).$$
Therefore, 
$$\vol(\Omega)^{\frac{1}{n}}\cdot\frac{\vol(\Omega_1)^{\frac{n-1}{n}}}{\vol(H)}
\le40^{n}e^{10(n-1)\sqrt{K}d}\sup_{0<r\le d}\bigg(\frac{r^n\int_0^d\sinh^{n-1}(\sqrt{K}t)dt}{\int_0^r\sinh^{n-1}(\sqrt{K}t)dt}\bigg)^{\frac{1}{n}}.$$
The last term on the right hand side has the estimate
$$\frac{r^n\int_0^d\sinh^{n-1}(\sqrt{K}t)dt}{\int_0^r\sinh^{n-1}(\sqrt{K}t)dt}\le r^n\cdot\frac{d}{r}\cdot\frac{\sinh^{n-1}(\sqrt{K}d)}
{\sinh^{n-1}(\sqrt{K}r)}\le d^n\cdot\frac{\sinh^{n-1}(\sqrt{K}d)}{(\sqrt{K}d)^{n-1}}\le d^ne^{(n-1)\sqrt{K}d}.$$
The required normalized Neumann isoperimetric constant estimate now follows.
\end{proof}

\end{document}